\documentclass{article}\usepackage{amsmath}

\usepackage{amssymb}
\usepackage{amsfonts}

\newtheorem{theorem}{Theorem}

\newenvironment{proof}[1][Proof]{\noindent\textbf{#1.} }{\ \rule{0.5em}{0.5em}}
\input{tcilatex}
\begin{document}

\title{Uniform Estimation Beyond the Mean}
\author{Andreas Maurer \\
Adalbertstr.55\\
D-80799 M\"{u}nchen\\
am@andreas-maurer.eu}
\maketitle

\begin{abstract}
Finite sample bounds on the estimation error of the mean by the empirical
mean, uniform over a class of functions, can often be conveniently obtained
in terms of Rademacher or Gaussian averages of the class. If a function of $n
$ variables has suitably bounded partial derivatives, it can be substituted
for the empirical mean, with uniform estimation again controlled by Gaussian
averages. Up to a constant the result recovers standard results for the
empirical mean and more recent ones about U-statistics, and extends to a
general class of estimation problems.
\end{abstract}

\section{Introduction}

Suppose we are given a class $\tciFourier $ of loss functions $f:\mathcal{%
X\rightarrow }\left[ 0,1\right] $, where $\mathcal{X}$ is some space, and a
vector of independent observations $\mathbf{X}=\left( X_{1},...,X_{n}\right) 
$, obeying some common law of probability $\mu $. The method of empirical
risk minimization seeks some $f\in \tciFourier $ which minimizes the
empirical average $\Phi \left( f\left( \mathbf{X}\right) \right) =\Phi
\left( f\left( X_{1}\right) ,...,f\left( X_{n}\right) \right) $, where%
\[
\Phi \left( s_{1},...,s_{n}\right) :=\frac{1}{n}\sum_{i=1}^{n}s_{i}\text{
for }s_{i}\in \left[ 0,1\right] .
\]%
The intuitive motivation of this method is the underlying hope that one
thereby approximately minimizes the expectation $\mathbb{E}_{\mathbf{X}%
^{\prime }}\Phi \left( f\left( \mathbf{X}^{\prime }\right) \right) =\mathbb{E%
}_{X\sim \mu }f\left( X\right) $ (where $\mathbf{X}^{\prime }$ is always iid
to $\mathbf{X}$). A fundamental problem in learning theory is the
justification of this hope in form of a uniform finite-sample bound of the
following type:

For every law $\mu $, every $n\in 
\mathbb{N}
$, and every $\delta >0$ there is a number $\mathcal{B}\left( \delta
,n\right) $ such that%
\begin{equation}
\Pr_{\mathbf{X}}\left\{ \sup_{f\in \tciFourier }\left( \mathbb{E}_{\mathbf{X}%
^{\prime }}\left[ \Phi \left( f\left( \mathbf{X}^{\prime }\right) \right) %
\right] -\Phi \left( f\left( \mathbf{X}\right) \right) \right) >\mathcal{B}%
\left( \delta ,n\right) \right\} <\delta .  \label{uniform law}
\end{equation}%
The bound $\mathcal{B}\left( \delta ,n\right) $ should depend little on the
confidence parameter $\delta $ and go to zero as $n\rightarrow \infty $.
This paper is motivated by the question under what conditions such bounds
can be found for other functions $\Phi $, beyond arithmetic means, such as
U-statistics or other, more general, nonlinear functions.\bigskip

One method to prove bounds of the form (\ref{uniform law}) above, which has
gained great popularity over the last decade and a half, is the method of
Rademacher and Gaussian averages (Kolchinskii 2000\nocite{Kolchinskii 2000},
Bartlett and Mendelson 2002\nocite{Bartlett 2002}). Given a subset $%
Y\subseteq 
\mathbb{R}
^{n}$ one defines 
\[
R\left( Y\right) =\mathbb{E}\sup_{\mathbf{y}\in Y}\sum_{i}\epsilon _{i}y_{i}%
\text{ and }G\left( Y\right) =\mathbb{E}\sup_{\mathbf{y}\in Y}\sum_{i}\gamma
_{i}y_{i}\text{,}
\]%
where the $\epsilon _{i}$ are independent uniform $\left\{ -1,1\right\} $%
-valued random variables and the $\gamma _{i}$ are independent standard
normal variables. The Rademacher averages $R\left( Y\right) $ and the
Gaussian averages $G\left( Y\right) $ are related by the inequalities $%
R\left( Y\right) \leq \sqrt{\pi /2}~G\left( Y\right) $ and $G\left( Y\right)
\leq 3\ln \left( n\right) R\left( Y\right) $ (see Ledoux and Talagrand 1991, 
\nocite{Ledoux 1991}). These quantities come into play as follows.

The random variable to bound is $\Psi \left( \mathbf{X}\right) =\sup_{f\in
\tciFourier }\left( \mathbb{E}\left[ \Phi \left( f\left( \mathbf{X}^{\prime
}\right) \right) \right] -\Phi \left( f\left( \mathbf{X}\right) \right)
\right) $. We write 
\[
\Psi \left( \mathbf{X}\right) =\mathbb{E}_{\mathbf{X}^{\prime }}\Psi \left( 
\mathbf{X}^{\prime }\right) +\left[ \Psi \left( \mathbf{X}\right) -\mathbb{E}%
_{\mathbf{X}^{\prime }}\Psi \left( \mathbf{X}^{\prime }\right) \right] .
\]%
The second term in this decomposition is the deviation of the random
variable $\Psi \left( \mathbf{X}\right) $ from its mean, and it can be
controlled using the well known bounded difference inequality (see McDiarmid
1998\nocite{McDiarmid 1998} or Boucheron et al 2013\nocite{Boucheron 2013},
Theorem \ref{Theorem bounded difference} below). The crucial property of the
arithmetic mean is that it changes little (here at most $1/n$) if only one
of its arguments is modified. The bounded difference inequality then gives a
bound of $\sqrt{\ln \left( 1/\delta \right) /\left( 2n\right) }$ with
probability at most $\delta $ for the second term. For the first term a
straightforward symmetrization argument gives the bound%
\[
\mathbb{E}_{\mathbf{X}}\Psi \left( \mathbf{X}\right) =\mathbb{E}_{\mathbf{X}%
}\sup_{f\in \tciFourier }\left( \mathbb{E}\left[ \Phi \left( f\left( \mathbf{%
X}^{\prime }\right) \right) \right] -\Phi \left( f\left( \mathbf{X}\right)
\right) \right) \leq \frac{2}{n}\mathbb{E}_{\mathbf{X}}\left[ R\left(
\tciFourier \left( \mathbf{X}\right) \right) \right] ,
\]%
where $\tciFourier \left( \mathbf{X}\right) =\left\{ f\left( \mathbf{X}%
\right) =\left( f\left( X_{1}\right) ,...,f\left( X_{n}\right) \right) :f\in
\tciFourier \right\} $ is a random subset of $%
\mathbb{R}
^{n}$. Since typically $R\left( \tciFourier \left( \mathbf{X}\right) \right) 
$ is of order $\sqrt{n}$ this term is also of order $1/\sqrt{n}$. Putting
the two bounds together gives (\ref{uniform law}) with 
\[
\mathcal{B}\left( \delta ,n\right) =\frac{2}{n}\mathbb{E}_{\mathbf{X}%
}R\left( \tciFourier \left( \mathbf{X}\right) \right) +\sqrt{\frac{\ln
\left( 1/\delta \right) }{2n}}.
\]%
Replacing the Rademacher average with the Gaussian average incurs only a
factor of $\sqrt{\pi /2}$. Both complexity measures have been very
successful, because they are often very easy to bound in practice.\bigskip

What properties of a general function $\Phi $ could guarantee similar
results? Clearly the same decomposition as above is always possible, and the
bounded difference inequality just requires that $\Phi $ changes only in the
order of $1/n$ if one of its arguments is modified. This concentration
property seems to be a very common-sense postulate, which we may retain as a
requirement for $\Phi $.

The difficulty still lies in the first term, because the usual
symmetrization argument relies heavily on the linearity of the arithmetic
mean. This suggests that we should get reasonable results if $\Phi $ is
'nearly' linear, in some sense of small curvature. The crucial requirement
is that the \textit{change} of $\Phi ,$ as one argument is changed, does not
depend too strongly on the \textit{other} arguments. We will formulate this
requirement in terms of mixed partial derivatives, which in (\ref{uniform
law}) will give us the bound%
\[
\mathcal{B}\left( \delta ,n\right) =c\left( L+M\right) \mathbb{E}_{\mathbf{X}%
}G\left( \tciFourier \left( \mathbf{X}\right) \right) +L\sqrt{n\ln \left(
1/\delta \right) /2},
\]%
where $c$ is a (unfortunately rather large) universal constant. Here the
bounded difference condition and our constraints on the mixed partial
derivatives of $\Phi $ are expressed in the quantities $L$ and $M$
respectively. For the arithmetic mean $L=1/n$ and $M=0$, so the price we pay
for the generality of $\Phi $ is the large constant and the presence of
Gaussian instead of the Rademacher average. This price is due to the use of
Talagrand's majorizing measure theorem, a powerful result, which was the
only working vehicle the author could find for the proof.

The first nontrivial cases are furnished by U-statistics, and we will see
that in this case $M$ and $L$ are of order $1/n$, so that we obtain bounds
of the same order as for the mean. It must at once be admitted that for
U-statistics such a result, with small constant and Rademacher instead of
Gaussian averages, has already been published by Clemencon et al (2008)%
\nocite{Clemencon Ustat}. Their method uses a trick introduced by Hoeffding
(1963)\nocite{Hoeffding 1963}, which reduces U-statistics to linear
functions. Nevertheless Hoeffding's method uses permutation arguments and
works only if the variables $X_{i}$ are identically distributed, while for
our method they only need to be independent. Besides this, U-statistics
possess a certain rigidity, while our result is applicable to a fairly large
class of functions $\Phi $. Generic members of this class have first partial
derivatives uniformly bounded in order of $1/n$ and mixed partial
derivatives uniformly bounded in order of $1/n^{2}$. These properties ensure 
$L$ and $M$ to be of order $1/n$.\bigskip

The next section introduces some necessary notation, states our main result
and sketches some applications. The last section is devoted to the proof of
our main result.

\section{Main results}

Before stating our result we introduce some notation: the letter $\mathcal{X}
$ always denotes some arbitrary set. If $F$ is a function on $\mathcal{X}^{n}
$ of $n$ variables, and $\mathbf{x}=\left( x_{1},...,x_{n}\right) \in 
\mathcal{X}^{n}$ we use $F_{k}\left( \mathbf{x},y\right) $ to denote $%
F\left( \mathbf{x}^{\prime }\right) $ where $x_{i}^{\prime }=x_{i}$ for $%
i\neq k$ and $x_{k}^{\prime }=y$. We use $\mathbf{e}_{1},...,\mathbf{e}_{n}$
to denote the canonical basis of $%
\mathbb{R}
^{n}$. If $F$ is a twice differentiable function of several real variables
then $\partial _{k}F$ is the partial derivative of $F$ w.r.t. the $k$-th
variable, and $\partial _{lk}F$ is the second partial derivative w.r.t. the $%
k$-th and $l$-th variable. For functions $F:\mathcal{X\rightarrow 
\mathbb{R}
}$ we write $\left\Vert F\right\Vert _{\infty }=\sup_{x\in \mathcal{X}%
}\left\vert f\left( x\right) \right\vert .$ The letter $c$ will always
denote a universal constant, which is allowed to be modified within proofs
from line to line in the standard way, so that, for example, $3c$ in one
line can become $c$ in the next line. If $\mathbf{X}$ is any random vector, $%
\mathbf{X}^{\prime }$ will always be iid to $\mathbf{X}$, which of course
does not mean that the components of $\mathbf{X}$ are iid.\bigskip 

\begin{theorem}
\label{Theorem Main} Let $\mathbf{X}=\left( X_{1},...,X_{n}\right) $ be a
vector of independent random variables with values in $\mathcal{X}$, $%
\mathbf{X}^{\prime }$ iid to $\mathbf{X}$, and let $\tciFourier $ be a
finite class of functions $f:\mathcal{X\rightarrow }\left[ 0,1\right] $ .
Assume $\Phi :%
\mathbb{R}
^{n}\rightarrow 
\mathbb{R}
$ to be twice differentiable, satisfying the conditions 
\begin{equation}
\forall k,\left\Vert \partial _{k}\Phi \right\Vert _{\infty }\leq L
\label{Coordinatewise Lipschitz}
\end{equation}%
and 
\begin{equation}
\sqrt{\sum_{k}\left\Vert \sum_{l:k\neq l}\left( \partial _{lk}\Phi \right)
^{2}\right\Vert _{\infty }}\leq M.  \label{Mixed deivative}
\end{equation}%
Then 
\begin{equation}
\mathbb{E}\sup_{f\in \tciFourier }\left[ \mathbb{E}\Phi \left( f\left( 
\mathbf{X}\right) \right) -\Phi \left( f\left( \mathbf{X}\right) \right) %
\right] \leq c\left( M+L\right) \mathbb{E}G\left( \tciFourier \left( \mathbf{%
X}\right) \right) .  \label{Main Inequality}
\end{equation}%
Furthermore, if $\delta >0$ then with probability at least $1-\delta $ in $%
\mathbf{X}$ it holds for all $f\in \tciFourier $ that%
\begin{equation}
\mathbb{E}\left[ \Phi \left( f\left( \mathbf{X}^{\prime }\right) \right) %
\right] \leq \Phi \left( f\left( \mathbf{X}\right) \right) +c\left(
L+M\right) \mathbb{E}_{\mathbf{X}}G\left( \tciFourier \left( \mathbf{X}%
\right) \right) +L\sqrt{\frac{n\ln \left( 1/\delta \right) }{2}/2}.
\label{generalization bound}
\end{equation}
\end{theorem}

Remarks:

1. Clearly condition (\ref{Mixed deivative}) is satisfied trivially with $M=0
$ for linear $\Phi $. In general, to have bounds of order $1/\sqrt{n}$ we
want both $M$ and $L$ to be of order $1/n$. This is guaranteed if the first
partial derivatives are of order $1/n$, and the mixed second partial
derivatives are order $1/n^{2}$.

2. Condition (\ref{Coordinatewise Lipschitz}) is what we need for the
application of the bounded difference inequality, and it will give us the
last term in the generalization bound (\ref{generalization bound}).

3. The condition (\ref{Mixed deivative}) is always satisfied if%
\[
\sqrt{\sum_{k,l:k\neq l}\left\Vert \partial _{lk}\Phi \right\Vert _{\infty
}^{2}}\leq M,
\]%
which is easier to verify. It may be that with a more careful analysis the
condition (\ref{Mixed deivative}) can be further relaxed to 
\[
\sqrt{\left\Vert \sum_{k,l:k\neq l}\left( \partial _{lk}\Phi \right)
^{2}\right\Vert _{\infty }}\leq M.
\]

4. It is evident from the proof, that the differentiability assumption can
be removed, if condition (\ref{Coordinatewise Lipschitz}) is replaced by the
requirement that $\Phi $ be $L$-Lipschitz in each coordinate separately, and
condition (\ref{Mixed deivative}) takes the form of a second order Lipschitz
condition. The statement of the latter condition however appears somewhat
cumbersome, so that here twice differentiability has been assumed for
greater clarity.

5. Other candidates for conditions on $\Phi $ come to mind, which would
allow similar results. A simple one is the requirement that $\Phi $ be a
Lipschitz function with respect to the euclidean distance on $%
\mathbb{R}
^{n}$. Unfortunately the Lipschitz constant of the arithmetic mean is
already $1/\sqrt{n}$, so with Rademacher or Gaussian averages being of order 
$\sqrt{n}$ no useful bounds result, not even in the simplest case.\bigskip 

We conclude this section with some simple examples. First consider the
sample variance given on $\left[ 0,1\right] ^{n}$ by%
\[
\Phi \left( \mathbf{s}\right) =\frac{1}{n\left( n-1\right) }\sum_{i<j}\left(
s_{i}-s_{j}\right) ^{2}.
\]%
Then 
\[
\partial _{k}\Phi \left( \mathbf{s}\right) =\frac{2}{n\left( n-1\right) }%
\sum_{i:i\neq k}\left( s_{k}-s_{i}\right) 
\]%
and for $l\neq k$%
\[
\partial _{lk}\Phi \left( \mathbf{s}\right) =\frac{-2}{n\left( n-1\right) },
\]%
from which we obtain $L=2/n$ and $M=2/\sqrt{n\left( n-1\right) }\leq
2/\left( n-1\right) $. The sample variance is a second order U-statistic
with kernel $\kappa \left( s,s^{\prime }\right) =\left( s-s^{\prime }\right)
^{2}/2$. 

Now consider the general U-statistic of $m$-th order%
\[
\Phi \left( \mathbf{s}\right) =\frac{1}{\binom{n}{m}}\sum_{i_{1}<\cdots
<i_{m}}\kappa \left( s_{i_{1}},...,s_{i_{m}}\right) ,
\]%
where $\kappa :\left[ 0,1\right] ^{m}$ is a symmetric, twice differentiable
kernel of $m$ variables. Then for $k\in \left\{ 1,...,n\right\} $%
\[
\left\vert \partial _{k}\Phi \left( \mathbf{s}\right) \right\vert \leq \frac{%
1}{\binom{n}{m}}\sum_{i_{1}<\cdots <i_{m}:k\in \left\{
i_{1},...,i_{j}\right\} }\left\vert \partial _{k}\kappa \left(
s_{i_{1}},...,s_{i_{m}}\right) \right\vert \leq \frac{m}{n}\left\Vert
\partial _{1}\kappa \right\Vert _{\infty },
\]%
and similarly for $l\neq k$%
\[
\left\vert \partial _{lk}\Phi \left( \mathbf{s}\right) \right\vert \leq 
\frac{m\left( m-1\right) }{n\left( n-1\right) }\left\Vert \partial
_{12}\kappa \right\Vert _{\infty },
\]%
so that $L$ and $M$ are again of order $1/n$. 

An example which is not a U-statistic and of practical relevance to learning
theory is constructed as follows. Let $\mu _{1},...,\mu _{K}$ be
distributions on $\mathcal{X}$ representing different classes of objects.
From each of the $\mu _{k}$ we draw an iid sample and let $\mathbf{X}$ be
the concatenation of these samples, where $\mathbf{X}$ has $n$ elements.
Observe that the $X_{i}$ and $X_{j}$ are not identically distributed. For $%
i,j\in \left\{ 1,...,n\right\} $ define $r_{ij}=1$ if $X_{i}$ and $X_{j}$
are drawn from the same distribution and $r_{ij}=-1$ if $X_{i}$ and $X_{j}$
are drawn from different distributions. Let $\tciFourier $ consist of
functions $f:\mathcal{X\rightarrow }\left[ 0,1\right] $. We seek a function $%
f\in \tciFourier $ which balances inter-class separation against intra-class
proximity. An obvious candidate is the functional $\mathbb{E}\Phi \left(
f\left( \mathbf{X}\right) \right) $ with%
\[
\Phi \left( \mathbf{s}\right) =\frac{1}{n\left( n-1\right) }%
\sum_{i<j}r_{ij}\left( s_{i}-s_{j}\right) ^{2}.
\]%
Except for the $r_{ij}$ this resembles the sample variance above, and it is
immediate that we obtain the same bounds for $M$ and $L$. On the other hand $%
\Phi $ is not permutation-symmetric nor are the $X_{i}$ identically
distributed.

\section{The proof}

We need two important auxiliary results. The first is the well known bounded
difference inequality, which goes back to Hoeffding (1963)\nocite{Hoeffding
1963} (see also McDiarmid 1998\nocite{McDiarmid 1998} and Boucheron et al
2013\nocite{Boucheron 2013}). Please recall the notation introduced at the
beginning of the previous section.

\begin{theorem}
\label{Theorem bounded difference}Suppose $F:\mathcal{X}^{n}\rightarrow 
\mathbb{R}
$ and $\mathbf{X}=\left( X_{1},...,X_{n}\right) $ is a vector of independent
random variables with values in $\mathcal{X}$, $\mathbf{X}^{\prime }$ is iid
to $\mathbf{X}$. Then 
\[
\Pr \left\{ F\left( \mathbf{X}\right) -\mathbb{E}F\left( \mathbf{X}^{\prime
}\right) >t\right\} \leq \exp \left( \frac{-2t^{2}}{\left\Vert \Delta
^{2}\right\Vert _{\infty }}\right) ,
\]%
where 
\[
\Delta ^{2}\left( \mathbf{x}\right) =\sum_{k=1}^{n}\sup_{y,z\in \mathcal{X}%
}\left( F_{k}\left( \mathbf{x},y\right) -F_{k}\left( \mathbf{x},z\right)
\right) ^{2}.
\]%
\bigskip 
\end{theorem}

The second auxiliary result is due to Michel Talagrand (see Theorem 15 in
Talagrand 1987\nocite{Talagrand 1987} or Theorem 2.1.5 in Talagrand 2005%
\nocite{Talagrand 2005}). It is a consequence of the celebrated majorizing
measure theorem (see e.g. Talagrand 1992\nocite{Talagrand 1992}). The
version we state is proved in (Maurer 2014\nocite{Maurer 2014}), adapted to
zero mean processes and $K=1$.

\begin{theorem}
\label{Theorem generic chaining}Let $X_{\mathbf{t}}$ be a random process
with zero mean, indexed by a finite set $T\subset 
\mathbb{R}
^{n}$. Suppose that for any distinct members $\mathbf{t},\mathbf{t}^{\prime
}\in Y$ and any $s>0$ 
\begin{equation}
\Pr \left\{ X_{\mathbf{t}}-X_{\mathbf{t}^{\prime }}>s\right\} \leq \exp
\left( \frac{-s^{2}}{2\left\Vert \mathbf{t}-\mathbf{t}^{\prime }\right\Vert
^{2}}\right)   \label{Process Tailbound}
\end{equation}%
Then 
\[
\mathbb{E}\sup_{\mathbf{t}\in T}X_{\mathbf{t}}\leq c~G\left( T\right) 
\]%
where $c$ is a universal constant.\bigskip 
\end{theorem}

The constant $c$ which results from the proof is of course very large (in
the hundreds). Nevertheless, as remarked in (Talagrand 1987\nocite{Talagrand
1987}), if $X$ is a Gaussian process, then Theorem \ref{Theorem generic
chaining} reduces to Slepian's Lemma (Boucheron et al 2013\nocite{Boucheron
2013}), which inspires the tantalizing conjecture that the optimal $c$ could
be in the order of unity, or even equal to one.\bigskip 

We are now prepared for the proof of Theorem \ref{Theorem Main}.

\begin{proof}[Proof of Theorem \protect\ref{Theorem Main}]
We first prove (\ref{Main Inequality}), the proof of the generalization
bound (\ref{generalization bound}) then being an easy application of the
bounded difference inequality.

Let $Q$ be the left hand side of (\ref{Main Inequality}). Initially our
proof parallels the standard symmetrization argument: we pull the second
expectation outside the supremum 
\[
Q\leq \mathbb{E}_{XX^{\prime }}\sup_{f\in \tciFourier }\left[ \Phi \left(
\sum_{i}f\left( X_{i}\right) \mathbf{e}_{i}\right) -\Phi \left(
\sum_{i}f\left( X_{i}^{\prime }\right) \mathbf{e}_{i}\right) \right] . 
\]%
Since $X_{i}$ and $X_{i}^{\prime }$ are iid, the last quantity does not
change if we exchange $X_{i}$ and $X_{i}^{\prime }$ on an arbirary subset of
indices $i$. If $\mathbf{\sigma }\in \left\{ 0,1\right\} ^{n}$ is such that $%
\sigma _{i}$ is zero on this set and one on its complement, we obtain%
\begin{eqnarray*}
Q &\leq &\mathbb{E}_{XX^{\prime }}\sup_{f\in \tciFourier }\left[ \Phi \left(
\sum_{i}\left[ \sigma _{i}f\left( X_{i}\right) +\left( 1-\sigma _{i}\right)
f\left( X_{i}^{\prime }\right) \right] \mathbf{e}_{i}\right) \right. \\
&&\left. -\Phi \left( \sum_{i}\left[ \sigma _{i}f\left( X_{i}^{\prime
}\right) +\left( 1-\sigma _{i}\right) f\left( X_{i}\right) \right] \mathbf{e}%
_{i}\right) \right] \\
&=&\mathbb{E}_{XX^{\prime }}\mathbb{E}_{\mathbf{\sigma }}\sup_{f\in
\tciFourier }\left[ \Phi \left( \sum_{i}\left[ \sigma _{i}f\left(
X_{i}\right) +\left( 1-\sigma _{i}\right) f\left( X_{i}^{\prime }\right) %
\right] \mathbf{e}_{i}\right) \right. \\
&&\left. -\Phi \left( \sum_{i}\left[ \sigma _{i}f\left( X_{i}^{\prime
}\right) +\left( 1-\sigma _{i}\right) f\left( X_{i}\right) \right] \mathbf{e}%
_{i}\right) \right] .
\end{eqnarray*}%
In the last step we took the expectation over configurations $\mathbf{\sigma 
}$ chosen uniformly from $\left\{ 0,1\right\} ^{n}$. We now condition on the 
$X_{i}$ and $X_{i}^{\prime }$ (which we temporarily replace by lower case
letters) and consider the random process%
\begin{eqnarray*}
Y_{f}\left( \mathbf{\sigma }\right) &=&\Phi \left( \sum_{i}\left[ \sigma
_{i}f\left( x_{i}\right) +\left( 1-\sigma _{i}\right) f\left( x_{i}^{\prime
}\right) \right] \mathbf{e}_{i}\right) \\
&&-\Phi \left( \sum_{i}\left[ \sigma _{i}f\left( x_{i}^{\prime }\right)
+\left( 1-\sigma _{i}\right) f\left( x_{i}\right) \right] \mathbf{e}%
_{i}\right) .
\end{eqnarray*}%
Clearly $\mathbb{E}_{\mathbf{\sigma }}Y_{f}\left( \mathbf{\sigma }\right) =0$
for all $f\in \tciFourier $.

Now we want to apply Theorem \ref{Theorem generic chaining}. To this end we
define a (pseudo-) metric on $\tciFourier $ by%
\[
d\left( f,g\right) =\left( \sum_{i=1}^{n}\left( f\left( x_{i}\right)
-g\left( x_{i}\right) \right) ^{2}+\left( f\left( x_{i}^{\prime }\right)
-g\left( x_{i}^{\prime }\right) \right) ^{2}\right) ^{1/2},f,g\in
\tciFourier 
\]%
and seek to prove, for fixed $f,g\in \tciFourier $ and $s>0$ the inequality%
\begin{equation}
\Pr \left\{ Y_{f}-Y_{g}>s\right\} \leq \exp \left( \frac{-s^{2}}{8\left(
M^{2}+L^{2}\right) d\left( f,g\right) ^{2}}\right) .  \label{Tailbound}
\end{equation}%
Let $Z\left( \mathbf{\sigma }\right) =Y_{f}\left( \mathbf{\sigma }\right)
-Y_{g}\left( \mathbf{\sigma }\right) $. To prove (\ref{Tailbound}) we will
apply the bounded difference inequality, Theorem \ref{Theorem bounded
difference}, to $Z$. Fix a configuration $\mathbf{\sigma \in }\left\{
0,1\right\} ^{n}$. We define the vectors $A,B,C,D\in \left[ 0,1\right] ^{n}$
by%
\begin{eqnarray*}
A &=&\sum_{i}\left( \sigma _{i}f\left( x_{i}\right) +\left( 1-\sigma
_{i}\right) f\left( x_{i}^{\prime }\right) \right) \mathbf{e}_{i} \\
B &=&\sum_{i}\left( \sigma _{i}g\left( x_{i}\right) +\left( 1-\sigma
_{i}\right) g\left( x^{\prime }\right) \right) \mathbf{e}_{i} \\
C &=&\sum_{i}\left( \sigma _{i}f\left( x_{i}^{\prime }\right) +\left(
1-\sigma _{i}\right) f\left( x\right) \right) \mathbf{e}_{i} \\
D &=&\sum_{i}\left( \sigma _{i}g\left( x_{i}^{\prime }\right) +\left(
1-\sigma _{i}\right) g\left( x\right) \right) \mathbf{e}_{i}.
\end{eqnarray*}%
Then for any $k\in \left\{ 1,...,n\right\} $%
\begin{align*}
& Z_{k}\left( \mathbf{\sigma },1\right) -Z_{k}\left( \mathbf{\sigma }%
,0\right)  \\
& =\Phi _{k}\left( A,f\left( x_{k}\right) \right) -\Phi _{k}\left( B,g\left(
x_{k}\right) \right) +\Phi _{k}\left( D,g\left( x_{k}^{\prime }\right)
\right) -\Phi _{k}\left( C,f\left( x_{k}^{\prime }\right) \right)  \\
& -\Phi _{k}\left( A,f\left( x_{k}^{\prime }\right) \right) +\Phi _{k}\left(
B,g\left( x_{k}^{\prime }\right) \right) -\Phi _{k}\left( D,g\left(
x_{k}\right) \right) +\Phi _{k}\left( C,f\left( x_{k}\right) \right) 
\end{align*}%
Adding and subtracting the quantities $\Phi _{k}\left( B,f\left(
x_{k}\right) \right) $, $\Phi _{k}\left( B,f\left( x_{k}^{\prime }\right)
\right) $, $\Phi _{k}\left( C,g\left( x_{k}^{\prime }\right) \right) $ and $%
\Phi _{k}\left( C,g\left( x_{k}\right) \right) $, rearranging terms, and
using Jensens inequality (which is responsible for the factor $1/8$) we get%
\begin{eqnarray}
&&\frac{1}{8}\left( Z_{k}\left( \mathbf{\sigma },1\right) -Z_{k}\left( 
\mathbf{\sigma },0\right) \right) ^{2}  \label{decomposition} \\
&\leq &\left[ \Phi _{k}\left( B,f\left( x_{k}\right) \right) -\Phi
_{k}\left( B,g\left( x_{k}\right) \right) \right] ^{2}+\left[ \Phi
_{k}\left( B,g\left( x_{k}^{\prime }\right) \right) -\Phi _{k}\left(
B,f\left( x_{k}^{\prime }\right) \right) \right] ^{2}  \nonumber \\
&&+\left[ \Phi _{k}\left( C,f\left( x_{k}\right) \right) -\Phi _{k}\left(
C,g\left( x_{k}\right) \right) \right] ^{2}+\left[ \Phi _{k}\left( C,g\left(
x_{k}^{\prime }\right) \right) -\Phi _{k}\left( C,f\left( x_{k}^{\prime
}\right) \right) \right] ^{2}  \nonumber \\
&&+\left[ \Phi _{k}\left( A,f\left( x_{k}\right) \right) -\Phi _{k}\left(
A,f\left( x_{k}^{\prime }\right) \right) -\left( \Phi _{k}\left( B,f\left(
x_{k}\right) \right) -\Phi _{k}\left( B,f\left( x_{k}^{\prime }\right)
\right) \right) \right] ^{2}  \nonumber \\
&&+\left[ \Phi _{k}\left( D,g\left( x_{k}^{\prime }\right) \right) -\Phi
_{k}\left( D,g\left( x_{k}\right) \right) -\left( \Phi _{k}\left( C,g\left(
x_{k}^{\prime }\right) \right) -\Phi _{k}\left( C,g\left( x_{k}\right)
\right) \right) \right] ^{2}  \nonumber
\end{eqnarray}%
The first four terms are controlled with the coordinatewise Lipschitz
condition (\ref{Coordinatewise Lipschitz}), and their sum is bounded by 
\begin{equation}
2L^{2}\left[ \left( f\left( x_{k}\right) -g\left( x_{k}\right) \right)
^{2}+\left( f\left( x_{k}^{\prime }\right) -g\left( x_{k}^{\prime }\right)
\right) ^{2}\right] .  \label{coordinates of dfg}
\end{equation}%
The last two terms are bounded using the condition (\ref{Mixed deivative})
on the mixed partials. Consider the term 
\[
T:=\left[ \Phi _{k}\left( A,f\left( x_{k}\right) \right) -\Phi _{k}\left(
A,f\left( x_{k}^{\prime }\right) \right) \right] -\left[ \Phi _{k}\left(
B,f\left( x_{k}\right) \right) -\Phi _{k}\left( B,f\left( x_{k}^{\prime
}\right) \right) \right] .
\]%
Define a function $F:\left[ 0,1\right] ^{2}\rightarrow 
\mathbb{R}
$ by%
\[
F\left( t,s\right) =\Phi _{k}\left( tA+\left( 1-t\right) B,sf\left(
x_{k}\right) +\left( 1-s\right) f\left( x_{k}^{\prime }\right) \right) .
\]%
Then%
\[
T=\left[ F\left( 1,1\right) -F\left( 1,0\right) \right] -\left[ F\left(
0,1\right) -F\left( 0,0\right) \right] =\int_{0}^{1}\int_{0}^{1}\partial
_{12}F\left( t,s\right) dsdt,
\]%
so that $T^{2}\leq \sup_{\,s,t\in \left[ 0,1\right] }\left[ \partial
_{12}F\left( t,s\right) \right] ^{2}$. Now%
\begin{eqnarray*}
\partial _{12}F\left( t,s\right)  &=&\sum_{l:l\neq k}\left( \partial
_{lk}\Phi _{k}\right) \left( tA+\left( 1-t\right) B,sf\left( x_{k}\right)
+\left( 1-s\right) f\left( x_{k}^{\prime }\right) \right)  \\
&&\times \left( f\left( x_{k}\right) -f\left( x_{k}^{\prime }\right) \right)
\left( A_{l}-B_{l}\right) ,
\end{eqnarray*}%
and, using $\left\vert f\left( x_{k}\right) -f\left( x_{k}^{\prime }\right)
\right\vert \leq 1$, Cauchy Schwarz, and the definitions of $A$ and $B$,%
\begin{eqnarray}
\sup_{\,s,t\in \left[ 0,1\right] }\partial _{12}F\left( t,s\right) ^{2}
&\leq &\left\Vert \left[ \sum_{l:l\neq k}\left( \partial _{lk}\Phi
_{k}\right) \left( A_{l}-B_{l}\right) \right] ^{2}\right\Vert _{\infty }
\label{Hessebound} \\
&\leq &\left\Vert \sum_{l:l\neq k}\left( \partial _{lk}\Phi _{k}\right)
^{2}\right\Vert _{\infty }\sum_{l:l\neq k}\left( A_{l}-B_{l}\right) ^{2} 
\nonumber \\
&\leq &\left\Vert \sum_{l:l\neq k}\left( \partial _{lk}\Phi _{k}\right)
^{2}\right\Vert _{\infty }d\left( f,g\right) ^{2}.  \nonumber
\end{eqnarray}%
The last term in (\ref{decomposition}) is bounded in exactly the same way.
Summing these bounds and the bound in (\ref{coordinates of dfg}) over $k$ we
get%
\[
\sum_{k}\left( Z_{k}\left( \mathbf{\sigma },1\right) -Z_{k}\left( \mathbf{%
\sigma },0\right) \right) ^{2}\leq 16\left( M^{2}+L^{2}\right) d\left(
f,g\right) ^{2}.
\]%
The bounded difference inequality then gives us 
\[
\Pr \left\{ Z>s\right\} \leq \exp \left( \frac{-2^{2}}{8\left(
M^{2}+L^{2}\right) d\left( f,g\right) ^{2}}\right) ,
\]%
which proves the desired (\ref{Tailbound}).

Now let $H_{f}$ be the process defined by $H_{f}=Y_{f}/\sqrt{4\left(
M^{2}+L^{2}\right) }$. Then 
\[
\Pr \left\{ H_{f}-H_{g}>s\right\} \leq \exp \left( \frac{-s^{2}}{2d\left(
f,g\right) ^{2}}\right) .
\]%
Since $d$ is exactly the euclidean metric on $\tciFourier \left( \mathbf{x},%
\mathbf{x}^{\prime }\right) \subseteq 
\mathbb{R}
^{2n}$ we can apply Theorem \ref{Theorem generic chaining} to $H_{f}$ and
conclude that%
\begin{eqnarray*}
\mathbb{E}\sup_{f}Y_{f} &=&\sqrt{4\left( M^{2}+L^{2}\right) }\mathbb{E}%
\left( \sup_{f}H_{f}-H_{f_{0}}\right)  \\
&\leq &c\sqrt{M^{2}+L^{2}}\mathbb{E}\sup_{f}\sum_{i}\left( \gamma
_{i}f\left( x_{i}\right) +\gamma _{i}^{\prime }f\left( x_{i}^{\prime
}\right) \right) .
\end{eqnarray*}%
We now remove the conditioning and return to the $X_{i}$-variables, to get%
\begin{eqnarray*}
Q &\leq &\mathbb{E}_{XX^{\prime }}\mathbb{E}_{\sigma }\sup_{f}Y_{f}\leq c%
\sqrt{M^{2}+L^{2}}\mathbb{E}_{XX^{\prime }}\mathbb{E}_{\gamma \gamma
^{\prime }}\sup_{f}\sum_{i}\left( \gamma _{i}f\left( X_{i}\right) +\gamma
_{i}^{\prime }f\left( X_{i}^{\prime }\right) \right)  \\
&\leq &c\sqrt{M^{2}+L^{2}}\mathbb{E}\sup_{f}\sum_{i}\gamma _{i}f\left(
X_{i}\right) .
\end{eqnarray*}%
This completes the proof of the first part of the theorem, inequality (\ref%
{Main Inequality}), because $\sqrt{M^{2}+L^{2}}\leq M+L$.

For the second assertion let $\Psi \left( \mathbf{X}\right) =\sup_{f\in
\tciFourier }\left( \mathbb{E}\left[ \Phi \left( f\left( \mathbf{X}^{\prime
}\right) \right) \right] -\Phi \left( f\left( \mathbf{X}\right) \right)
\right) $ and write, just as in the introduction,%
\begin{equation}
\Psi \left( \mathbf{X}\right) =\mathbb{E}\left[ \Psi \left( \mathbf{X}%
^{\prime }\right) \right] +\left( \Psi \left( \mathbf{X}\right) -\mathbb{E}%
\left[ \Psi \left( \mathbf{X}^{\prime }\right) \right] \right) .
\label{Deco}
\end{equation}%
The first term has already been bounded in (\ref{Main Inequality}). For the
second term observe that, since the functions in $\tciFourier $ have range
in $\left[ 0,1\right] $, $\Psi \left( \mathbf{X}\right) $ changes at most by 
$L$ if any of its arguments is modified. The bounded difference inequality
gives 
\[
\Pr \left\{ \Psi \left( \mathbf{X}\right) -\mathbb{E}\left[ \Psi \left( 
\mathbf{X}^{\prime }\right) \right] >t\right\} \leq \exp \left( \frac{-2t^{2}%
}{nL^{2}}\right) . 
\]%
Equating to $\delta $ and solving for $t$ gives with probability at least $%
1-\delta $ that%
\[
\Psi \left( \mathbf{X}\right) -\mathbb{E}\left[ \Psi \left( \mathbf{X}%
^{\prime }\right) \right] \leq L\sqrt{\frac{n\ln \left( 1/\delta \right) }{2}%
}. 
\]%
Together with the decomposition (\ref{Deco}) and the bound on $\mathbb{E}%
\left[ \Psi \left( \mathbf{X}\right) \right] $ implied by (\ref{Main
Inequality}) this completes the proof of the generalization bound (\ref%
{generalization bound}).$\blacksquare $
\end{proof}

\end{document}